\documentclass[12pt]{amsart}
\usepackage{amsmath,amsthm}
\usepackage{amssymb,latexsym}
\usepackage{enumerate}
\usepackage[T1]{fontenc}
\usepackage{array}
\usepackage[all]{xypic}
\usepackage{relsize}
\usepackage[left=3.0cm,right=3.0cm,top=3.0cm,bottom=3.0cm]{geometry}

\usepackage{thmtools}

\newtheorem{thm}{Theorem}[section]

\newtheorem{lem}[thm]{Lemma}
\newtheorem{prop}[thm]{Proposition}

\theoremstyle{definition}
\newtheorem{exa}[thm]{Example}

\newtheorem*{thm1}{Theorem}

\renewcommand*{\mod}{\mathrm{mod}\,}
\newcommand{\ind}{\mathrm{ind}\,}
\newcommand{\rad}{\mathrm{rad}\,}
\renewcommand*{\dim}{\mathrm{dim}}

\newcommand{\Tr}{\mathrm{Tr}}

\newcommand{\Hom}{\mathrm{Hom}}
\newcommand{\End}{\mathrm{End}}
\newcommand{\op}{\mathrm{op}}
\newcommand{\D}{\mathrm{D}}

\newcommand{\soc}{\mathrm{soc}}

\newcommand{\supp}{\mathrm{supp} }
\renewcommand*{\top}{\mathrm{top}}
\renewcommand*{\Im}{\mathrm{Im}\,}

\begin{document}

\title[]{Selfinjective algebras without short cycles of
indecomposable modules}
\maketitle \vspace{0.6cm}
\begin{center}
ALICJA JAWORSKA-PASTUSZAK and ANDRZEJ SKOWRO\'NSKI \vspace{1cm}

 {\em Dedicated to Zygmunt Pogorza{\l}y on the
occasion of his 60th birthday} \end{center} \vspace{0.5cm}
\begin{abstract}
We describe the structure of finite dimensional selfinjective algebras over an arbitrary field without short cycles of indecomposable modules.
\end{abstract}
\vspace{0.5cm}

\section*{Introduction and the main result.}
Throughout the paper, by an algebra we mean a basic indecomposable
finite dimensional  associative $K$-algebra with identity over a
field $K$. For an algebra $A$, we denote by $\mod A$ the category
of finite-dimensional right $A$-modules, by $\ind A$ the full
subcategory of $\mod A$ formed by the indecomposable modules, by
$\Gamma_A$ the Auslander-Reiten quiver of $A$, and by $\tau_A$ and
$\tau_A^{-1}$ the Auslander-Reiten translations $\D\Tr$ and
$\Tr\D$, respectively. We do not distinguish between a module in
$\ind A$ and the vertex of $\Gamma_A$ corresponding to it. An
algebra $A$ is of finite representation type if the category $\ind
A$ admits only a finite number of pairwise nonisomorphic modules.
It is well known that  a hereditary algebra $A$ is of finite
representation type if and only if $A$ is of Dynkin type, that is,
the valued quiver $Q_A$ of $A$ is a Dynkin quiver of type
$\mathbb{A}_n$ ($n \geqslant 1$), $\mathbb{B}_n$ ($n\geqslant 2$),
$\mathbb{C}_n$ ($n\geqslant 3$), $\mathbb{D}_n$ ($n\geqslant 4$),
$\mathbb{E}_6$, $\mathbb{E}_7$, $\mathbb{E}_8$, $\mathbb{F}_4$ and
$\mathbb{G}_2$ (see \cite{DR1}, \cite{DR2}, \cite{G1}). A
distinguished class of algebras of finite representation type is
formed by the tilted algebras of Dynkin type, that is, the
algebras of the form $\End_H(T)$ for a hereditary algebra $H$ of
Dynkin type and a (multiplicity-free) tilting module $T$ in $\mod
H$. An algebra $A$ is called selfinjective if $A_A$ is an
injective module, or equivalently, the projective modules in $\mod
A$ are injective. For a selfinjective algebra $A$, we denote by
$\Gamma^s_A$ the stable Auslander-Reiten quiver of $A$, obtained
from $\Gamma_A$ by removing the projective modules and the arrows
attached to them.

We are concerned with the problem of describing the isomorphism
classes of selfinjective  algebras of finite representation type.
For $K$ algebraically closed, the problem was solved in the 1980's
by C. Riedtmann (see \cite{BLR}, \cite{Rd1}, \cite{Rd2},
\cite{Rd3}) via the combinatorial classification of the
Auslander-Reiten quivers of selfinjective algebras of finite
representation type. Equivalently, Riedtmann's classification can
be presented as follows (see \cite[Section 3]{S}): a nonsimple
selfinjective algebra $A$ over an algebraically closed field $K$
is of finite representation type if and only if $A$ is a socle
(geometric) deformation of an orbit algebra $\widehat{B}/G$, where
$\widehat{B}$  is the repetitive category of a tilted algebra $B$
of Dynkin type $\mathbb{A}_n$ ($n \geqslant 1$),  $\mathbb{D}_n$
($n\geqslant 4$), $\mathbb{E}_6$, $\mathbb{E}_7$, $\mathbb{E}_8$,
and $G$ is an admissible infinite cyclic group of automorphisms of
$\widehat{B}$. For an arbitrary field $K$, the problem seems to be
difficult (see \cite{BS1}, \cite{BS2}, \cite{Ho}, \cite{SY5} for
some results in this direction and \cite[Section 12]{SY6} for
related open problems). An important known result towards solution
of this general problem is the description of the stable
Auslander-Reiten quiver $\Gamma^s_A$ of a selfinjective algebra
$A$ of finite representation type established  by C. Riedtmann
\cite{Rd1} and G. Todorov \cite{T} (see also \cite[Section
IV.15]{SY7}): $\Gamma^s_A$ is isomorphic to the orbit quiver
$\mathbb{Z}\Delta / G$, where $\Delta$ is a Dynkin quiver of type
$\mathbb{A}_n$ ($n \geqslant 1$), $\mathbb{B}_n$ ($n\geqslant 2$),
$\mathbb{C}_n$ ($n\geqslant 3$), $\mathbb{D}_n$ ($n\geqslant 4$),
$\mathbb{E}_6$, $\mathbb{E}_7$, $\mathbb{E}_8$, $\mathbb{F}_4$ or
$\mathbb{G}_2$, and $G$ is an admissible infinite cyclic group of
automorphisms of the translation quiver $\mathbb{Z}\Delta$.
Therefore, we may associate to a selfinjective algebra $A$ of
finite representation type a Dynkin graph $\Delta(A)$, called the
Dynkin type of $A$, such that $\Gamma^s_A=\mathbb{Z}\Delta / G$
for a quiver $\Delta$ having $\Delta(A)$ as underlying graph. We
also mention that, for a tilted algebra $B$ of Dynkin type
$\Delta$, the orbit algebras $\widehat{B}/ G$  are selfinjective
algebras of finite representation type whose Dynkin type is the
underlying graph of $\Delta$.

Following \cite{RSS}, a {\it short cycle} in the module category
$\mod A$ of an algebra $A$ is a sequence
$$\xymatrix{X \ar[r]^f & Y \ar[r]^g  & X}$$
of two nonzero nonisomorphisms  between modules $X$ and $Y$ in
$\ind A$. It has been proved in \cite[Corollary 2.2]{RSS} that if
$M$ is a module in $\ind A$ which does not lie on a short cycle
then $M$ is uniquely determined (up to isomorphism) by its image
$[M]$ in the Grothendieck group $K_0(A)$. Moreover, by a result of
Happel and Liu \cite[Theorem]{HL} every algebra $A$ having no
short cycles in $\mod A$ is of finite representation type.

The following theorem is the main result of the paper.
\begin{thm1}
{ \it Let $A$ be a selfinjective algebra over a field $K$. The
following statements are equivalent.
\begin{enumerate}[\rm (i)]
\item $\mod A$ has no short cycles.
\item $A$ is isomorphic to an orbit algebra $\widehat{B}/ (\varphi \nu^2_{\widehat{B}})$,
where  $B$ is a tilted algebra of Dynkin type over $K$,
$\nu_{\widehat{B}}$ is the Nakayama automorphism of $\widehat{B}$,
and $\varphi$ is a strictly positive automorphism of
$\widehat{B}$.
\end{enumerate}}
\end{thm1}

The paper is organized as follows. In Section 1 we introduce the
orbit algebras of repetitive categories. Section 2 is devoted to
presenting basic results from the theory of selfinjective algebras
with deforming ideals, playing the fundamental role in the proof
of the main theorem. In Section 3 we discuss properties of stable
slices of Auslander-Reiten quivers of selfinjective algebras of
finite type, essential for further considerations. In Section 4 we
describe the selfinjective Nakayama algebras without short cycles
of indecomposable modules. In Section 5 we prove the Theorem for
the selfinjective algebras of Dynkin type. In the final Section 6
we complete the proof of the Theorem for arbitrary selfinjective
algebras.

For basic background on the representation theory applied in this
paper we refer to \cite{ASS}, \cite{ARS}, \cite{SY7} and
\cite{SY7a}.

\section{Orbit algebras of repetitive categories.}

Let $B$ be an algebra and $1_{B}=e_{1}+\cdots +e_{n}$ a
decomposition of the identity of $B$ into a sum of pairwise
orthogonal primitive idempotents. We associate to $B$ a
selfinjective locally bounded $K$-category $\widehat B$, called
the \emph{repetitive category} of $B$ (see~\cite{HW}). The objects
of $\widehat B$ are $e_{m,i}$, for $m\in{\mathbb{Z}}$, $i\in \{1,
\dots, n\}$, and the morphism spaces are defined as follows
\begin{displaymath}
\widehat B(e_{m,i},e_{r,j})=\left\{\begin{array}{ll}
e_{j}Be_{i},    & r=m,\\
D(e_{i}Be_{j}),& r=m+1,\\
0,& \textrm{otherwise}.
\end{array} \right.
\end{displaymath}
Observe that $e_{j}Be_{i}=\Hom _{B}(e_{i}B,e_{j}B)$, $D(e_{i}Be_{j})=e_{j}D(B)e_{i}$ and
\begin{displaymath}
\bigoplus_{(m,i)\in{\mathbb{Z}\times \{1, \dots ,n\}}} \widehat B(e_{m,i},e_{r,j})=e_{j}B \oplus D(Be_{j}),
\end{displaymath}
for any $r\in{\mathbb{Z}}$ and $j\in\{1, \dots ,n\}$.
We denote by $\nu_{\widehat B}$ the \emph{Nakayama automorphism} of $\widehat B$ defined by
\begin{displaymath}
\nu_{\widehat B}(e_{m,i})=e_{m+1,i} \quad \textrm{for all} \quad (m,i)\in \mathbb{Z}\times\{1, \dots ,n\}.
\end{displaymath}
Moreover, an automorphism $\varphi$ of the $K$-category $\widehat
B$ is said to be:
\begin{itemize}\setlength{\parskip}{0pt}%
\setlength{\itemsep}{2pt}%
 \item \emph{positive} if, for each pair $(m,i)\in{\mathbb{Z}\times \{1, \dots ,n\}}$, we have $\varphi(e_{m,i})=e_{p,j}$ for some $p\geq m$ and some $j\in\{1, \dots ,n\}$;
 \item \emph{rigid} if, for each pair $(m,i)\in{\mathbb{Z}\times \{1, \dots ,n\}}$, there exists $j\in\{1, \dots ,n\}$ such that $\varphi(e_{m,i})=e_{m,j}$;
 \item \emph{strictly positive} if it is positive but not rigid.
\end{itemize}
Thus the automorphisms $\nu^{r}_{\widehat B}$, for any $r\geq 1$,
are strictly positive automorphisms of $\widehat B$.

Recall that a group $G$ of automorphisms  of $\widehat B$ is said
to be \emph{admissible} if  $G$ acts freely on the set of objects
of $\widehat B$ and has finitely many orbits.
 Then, following P.~Gabriel \cite{G2}, we may consider the orbit category
 $\widehat B/G$ of  $\widehat B$ with respect to $G$ whose objects are the $G$-orbits of objects in $\widehat B$,
 and the morphism spaces are given by
\begin{displaymath}
(\widehat B/G)(a,b)=
\Big\lbrace (f_{y,x})\in{\prod_{(x,y)\in{a\times b}}} \widehat B(x,y)\hspace{2mm}|\hspace{2mm} gf_{y,x}=f_{gy,gx}, \forall_{g\in{G}, (x,y)\in{a\times b}}\Big\rbrace
\end{displaymath}
for all objects $a,b$ of $\widehat B/G$. Since $\widehat B/G$ has
finitely many objects and the morphism spaces in $\widehat B/G$
are finite dimensional, we have the associated finite dimensional selfinjective $K$-algebra $\bigoplus (\widehat B/G)$ which is the
direct sum of all morphism spaces in $\widehat B/G$, called the
\emph{orbit algebra} of $\widehat B$ with respect to $G$. We will
identify $\widehat B/G$ with $\bigoplus (\widehat B/G)$. For
example, for each positive integer $r$, the infinite cyclic group
$(\nu^{r}_{\widehat B})$ generated by the $r$-th power
$\nu^{r}_{\widehat B}$ of $\nu_{\widehat B}$ is an admissible
group of automorphisms of $\widehat B$, and we have the associated
selfinjective orbit algebra

\begin{displaymath}
T(B)^{(r)}=\widehat B/(\nu^{r}_{\widehat
B})=\left\{\begin{array}{cc} \left[\smallmatrix
b_{1} & 0 & 0 &\cdots & 0 & 0 & 0 \\
f_{2} & b_{2} & 0 &\cdots & 0 & 0 & 0 \\
0 & f_{3} & b_{3} &\cdots & 0 & 0 & 0 \\
\vdots &\vdots &\ddots &\ddots &\vdots &\vdots &\vdots \\
\vdots &\vdots &\vdots &\ddots &\ddots &\vdots &\vdots \\
0 & 0 & 0 &\cdots & f_{r-1} & b_{r-1} & 0 \\
0 & 0 & 0 &\cdots & 0 & f_{1} & b_{1}
\endsmallmatrix \right]; &
 \begin{matrix} b_{1}, \ldots ,b_{r-1}\in{B},\\   f_{1}, \ldots ,f_{r-1}\in
 D(B)\end{matrix}
\end{array}\right\},
\end{displaymath}
called the $r$-\emph{fold trivial extension algebra of B}. In
particular,  $T(B)^{(1)}\cong T(B)=B\ltimes D(B)$ is the trivial
extension of $B$ by the injective cogenerator $D(B)$.

\section{Selfinjective algebras with deforming ideals.}

In this section we present criteria for selfinjective algebras to
be socle equivalent to orbit  algebras of the repetitive
categories of algebras with respect to infinite cyclic
automorphism groups, playing fundamental role in the proof of the
Theorem.

Let $A$ be a selfinjective algebra. For a subset $X$ of $A$, we
may consider the left annihilator  $l_{A}(X)=\{a\in
A\hspace{2mm}|\hspace{2mm}aX=0\}$ of $X$ in $A$ and the right
annihilator $r_{A}(X)=\{a\in A\hspace{2mm}|\hspace{2mm}Xa=0\}$ of
$X$ in $A$. Then by a theorem due to T.~Nakayama
(see~\cite[Theorem~IV.6.10]{SY7}) the annihilator operation
$l_{A}$ induces a Galois correspondence from the lattice of right
ideals of $A$ to the lattice of left ideals of $A$, and $r_{A}$ is
the inverse Galois correspondence to $l_{A}$. Let $I$ be an ideal
of $A$, $B=A/I$, and $e$ an idempotent of $A$ such that $e+I$ is
the identity of $B$. We may assume that $1_{A}=e_{1}+\cdots
+e_{r}$ with $e_{1},\ldots ,e_{r}$ pairwise orthogonal primitive
idempotents of $A$, $e=e_{1}+\cdots +e_{n}$ for some $n\leq r$,
and $\{e_{i}\hspace{2mm}|\hspace{2mm}1\leq i\leq n\}$ is the set
of all idempotents in $\{e_{i}\hspace{2mm}|\hspace{2mm}1\leq i\leq
r\}$ which are not in $I$. Then such an idempotent $e$ is uniquely
determined by $I$ up to an inner automorphism of $A$, and is
called a \emph{residual identity} of $B=A/I$. Observe also that
$B\cong eAe/eIe$.

We have the following lemma from~\cite[Lemma~5.1]{SY4}.
\begin{lem}\label{2.1}
Let $A$ be a selfinjective algebra, $I$ an ideal of  $A$, and $e$
an idempotent of $A$ such that $l_{A}(I)=Ie$ or $r_{A}(I)=eI$.
Then $e$ is a residual identity of $A/I$.
\end{lem}

We recall also the following proposition proved in \cite[Proposition~2.3]{SY1}.
\begin{prop}\label{2.2}
Let $A$ be a selfinjective algebra, $I$ an ideal of $A$, $B=A/I$, $e$ a residual identity of $B$,
and assume that $IeI=0$. The following conditions are equivalent.
\begin{enumerate}[\upshape (i)]
\item $Ie$ is an injective cogenerator in $\mod B$.
\item $eI$ is an injective cogenerator in $\mod B^{\op}$.
\item $l_{A}(I)=Ie$.
\item $r_{A}(I)=eI$.
\end{enumerate}
Moreover, under these equivalent conditions, we have $\soc (A)
\subseteq I$ and $l_{eAe}(I)=eIe=r_{eAe}(I)$.
\end{prop}

The following theorem proved in \cite[Theorem~3.8]{SY2}
(sufficiency part) and~\cite[Theorem~5.3]{SY4}  (necessity part)
will be fundamental for our considerations.

\begin{thm}\label{2.3}
Let $A$ be a selfinjective algebra. The following conditions are equivalent.
\begin{enumerate}[\upshape (i)]
\item $A$ is isomorphic to an orbit algebra $\widehat B/(\varphi \nu_{\widehat B})$, where $B$
is an algebra and $\varphi$ is a positive automorphism of $\widehat B$.
\item There is an ideal $I$ of $A$ and an idempotent $e$ of $A$ such that
\begin{enumerate}[\upshape (1)]
\item $r_{A}(I)=eI$;
\item the canonical algebra epimorphism $eAe\to eAe/eIe$ is a retraction.
\end{enumerate}
\end{enumerate}
Moreover, in this case, $B$ is isomorphic to $A/I$.
\end{thm}

Let $A$ be a selfinjective algebra, $I$ an ideal of $A$, and $e$ a
residual identity  of $A/I$. Following \cite{SY1}, $I$ is said to
be a \emph{deforming ideal} of $A$ if the following conditions are
satisfied:
\begin{enumerate}[\upshape (D1)]
\item[(D1)] $l_{eAe}(I)=eIe=r_{eAe}(I)$;
\item[(D2)] the valued quiver $Q_{A/I}$ of $A/I$ is acyclic.
\end{enumerate}
Assume now that $I$ is a deforming ideal of $A$. Then we have a
canonical isomorphism of  algebras $eAe/eIe\to A/I$ and $I$ can be
considered as an $(eAe/eIe)$-$(eAe/eIe)$-bimodule. Denote by
$A[I]$ the direct sum of $K$-vector spaces $(eAe/eIe)\oplus I$
with the multiplication
\begin{displaymath}
(b,x)\cdot (c,y)=(bc,by+xc+xy)
\end{displaymath}
for $b,c\in{eAe/eIe}$ and $x,y\in I$. Then $A[I]$ is a $K$-algebra
with the identity $(e+eIe,1_{A}-e)$,  and, by identifying
$x\in{I}$ with $(0,x)\in{A[I]}$,  we may consider $I$ as an ideal
of $A[I]$. Observe that $e=(e+eIe,0)$ is a residual identity of
$A[I]/I=eAe/eIe\cong A/I$, $eA[I]e=(eAe/eIe)\oplus eIe$, and the
canonical  algebra epimorphism $eA[I]e\to eA[I]e/eIe$ is a
retraction.

The following properties of the algebra $A[I]$ were established in \cite[Theorem~4.1]{SY1}.

\begin{thm}\label{2.4}
Let $A$ be a selfinjective algebra and $I$ a deforming ideal of $A$. The following statements hold.
\begin{enumerate}[\upshape (i)]
\item $A[I]$ is a selfinjective algebra with the same Nakayama permutation as $A$ and $I$ is a deforming ideal of $A[I]$.
\item $A$ and $A[I]$ are socle equivalent.
\end{enumerate}
\end{thm}
We note that if $A$ is a selfinjective algebra, $I$ an ideal of
$A$, $B=A/I$, $e$ an idempotent of  $A$ such that $r_{A}(I)=eI$,
and the valued quiver $Q_{B}$ of $B$ is acyclic, then by Lemma
\ref{2.1} and Proposition \ref{2.2}, $I$ is a deforming ideal of
$A$ and $e$ is a residual identity of $B$.

The following theorem proved in~\cite[Theorem~4.1]{SY2} shows the importance of the algebras  $A[I]$.

\begin{thm}\label{2.5}
Let $A$ be a selfinjective algebra, $I$ an ideal of $A$, $B=A/I$
and $e$ an idempotent of $A$.  Assume that $r_{A}(I)=eI$ and
$Q_{B}$ is acyclic. Then $A[I]$ is isomorphic to an orbit algebra
$\widehat B/(\varphi\nu_{\widehat B})$ for some positive
automorphism $\varphi$ of $\widehat B$.
\end{thm}

We point out that there are selfinjective algebras $A$ with
deforming ideals $I$ such  that the algebras $A$ and $A[I]$ are
not isomorphic (see~\cite[Example~4.2]{SY2}). The following
criterion proved in~\cite[Proposition~3.2]{SY3} describes a
situation  when the algebras $A$ and $A[I]$ are isomorphic.

\begin{thm}\label{2.6}
Let $A$ be a selfinjective algebra with a deforming ideal $I$,
$B=A/I$, $e$ be a residual identity of $B$, and $\nu$ the Nakayama
permutation of $A$.  Assume that $IeI=0$ and $e_{i}\neq
e_{\nu(i)}$, for any primitive summand $e_{i}$ of $e$. Then the
algebras $A$ and $A[I]$ are isomorphic. In particular, $A$ is
isomorphic to an orbit algebra $\widehat B/(\varphi\nu_{\widehat
B})$ for some positive automorphism $\varphi$ of $\widehat B$.
\end{thm}

\section{Selfinjective algebras of finite representation type.}

Let $A$ be a selfinjective algebra of finite representation type.
We know from the  Riedtmann-Todorov theorem that the stable
Auslander-Reiten quiver $\Gamma^s_A$ of $A$ is of the form
$\mathbb{Z}\Delta /G$ for a valued Dynkin quiver $\Delta$ and an
admissible infinite cyclic group $G$ of automorphisms of the
translation quiver $\mathbb{Z}\Delta$. Following \cite{SY8}, a
full valued subquiver $\Delta$ of $\Gamma_A$ is said to be a
\emph{stable slice} if the following conditions are satisfied:
\begin{enumerate}[\rm (1)]
\item $\Delta$ is connected, acyclic and without projective modules.
\item For any valued arrow $\xymatrix{V \ar[r]^{(d, d')} & U}$ in $\Gamma_A$
 with $U$ in $\Delta$ and $V$ nonprojective, $V$ belongs to $\Delta$ or to $\tau_A\Delta$.
\item  For any valued arrow $\xymatrix{U \ar[r]^{(e, e')} & V}$ in $\Gamma_A$
 with $U$ in $\Delta$ and $V$ nonprojective, $V$ belongs to $\Delta$ or to $\tau^{-1}_A\Delta$.
\end{enumerate}
Observe that a stable slice $\Delta$ of $\Gamma_A$ is a Dynkin
quiver intersecting  every $\tau_A$-orbit in $\Gamma^s_{A}$
exactly once. A stable slice $\Delta$ of $\Gamma_A$ is said to be
\emph{semiregular} if $\Delta$ does not contain both the socle
factor $Q /\soc (Q)$ of an indecomposable projective module $Q$
and the radical $ \rad P$ of an indecomposable projective module
$P$. Further, following \cite{SY8}, a stable slice $\Delta$ of
$\Gamma_A$ is said to be \emph{double} $\tau_A$-\emph{rigid} if
$\Hom_A(X, \tau_AY)=0$ and $\Hom_A(\tau^{-1}_AX,Y)=0$ for all
indecomposable modules $X$ and $Y$ from $\Delta$.

Recall also that a selfinjective algebra $A$ is a \emph{Nakayama
algebra}  if every indecomposable projective module $P$ in $\mod
A$ is uniserial (its submodule lattice is a chain). We note that
then $A$ is of finite representation type and every indecomposable
module  in $\mod A$ is uniserial (see \cite[Theorem I.10.5]{SY7}).
\begin{thm} \label{4.1}
Let $A$ be a selfinjective algebra of finite representation type.
The following statements are equivalent.
\begin{enumerate}[\rm (i)]
\item $\Gamma_A$ admits a semiregular stable slice.
\item $A$ is not a Nakayama algebra.
\end{enumerate}
\end{thm}
\begin{proof}
Let $A$ be an indecomposable selfinjective Nakayama algebra and
$n$ be the rank of $K_0(A)$. Note that then each $\tau_A$-orbit of
$\Gamma_A$ consists of $n$ indecomposable modules having the same
length (see \cite[ Corollary V.4.2]{ASS} or \cite[Corollary
IV.2.9]{ARS}). Therefore, a $\tau_A$-orbit in $\Gamma_A$ which
contains the radical of an indecomposable projective $A$-module
consists entirely of the radicals of all indecomposable projective
$A$-modules. Hence, no stable slice of $\Gamma_A$ is semiregular,
and (i) implies (ii).

Assume $A$ is not a Nakayama algebra. It implies that there exists
a projective module $P$ in $\ind A$ such that $P/\soc (P)$ is not
a radical of any indecomposable projective module. Indeed, if
$P_1, \ldots, P_n$ is a complete family of indecomposable
projective modules in $\mod A$ and $P_k/\soc (P_k)=\rad P_{k+1}$
for $k \in \{1, \ldots, n\}$, with $P_{n+1}=P_1$, then $P_1,
\ldots, P_n$ are uniserial modules, and $A$ is a Nakayama algebra,
a contradiction. We denote by $\Delta_P$ the full subquiver of
$\Gamma_A$ given by the module $\tau^{-1}_A(P/\soc (P))$ and all
modules $X$ in $\ind A$ such that there is a nontrivial sectional
path in $\Gamma_A^s$ from $P/\soc (P)$ to $X$. Observe that
$\Delta_P$ is a stable slice in $\Gamma_A$. We shall show that
$\Delta_P$ does not contain $Q/ \soc (Q)$ for any indecomposable
projective module $Q$ in $\mod A$, and hence it is semiregular.

Suppose, on the contrary, that $\Delta_P$ contains $Q/\soc (Q)$
for some indecomposable projective module $Q$ in $\mod A$. From
the assumption imposed on $P$ we know that $\tau^{-1}_A(P/\soc
(P)) \neq Q/ \soc (Q)$. Then $\Gamma_A$ contains a full valued
subquiver of the form
 \xymatrixrowsep{0.1cm} \xymatrixcolsep{0.5cm}
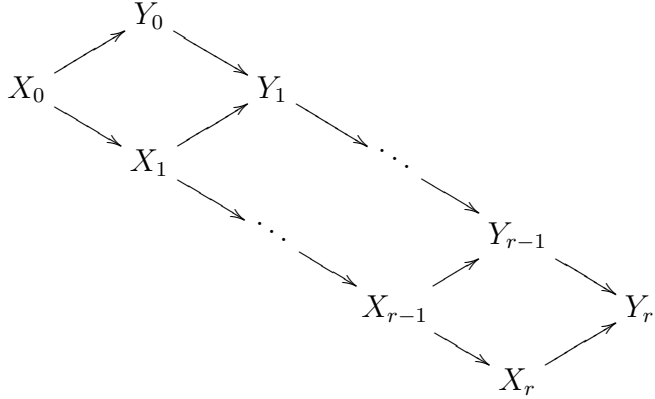
\[\xymatrix@!R@!C{ &  Y_0 \ar[rd] & &  & & &\\
X_0 \ar[ur] \ar[dr] & & Y_1 \ar[rd]  & & &\\
&  X_1 \ar[rd] \ar[ru] &  &{\ddots}\ar[rd]& \\
&  &\ddots \ar[rd]& &  Y_{r-1} \ar[rd]&\\
&&&X_{r-1}\ar[rd] \ar[ru]& & Y_r\\
&&&& X_r \ar[ru] & }\]  where $Y_0=P/\soc (P)$, all modules $Y_i$
for $i \in \{1, \ldots,   r\}$ belong to $\Delta_P$, $X_r$ is the
indecomposable projective module $Q$, $Y_r=Q/\soc (Q)$ and $r \geq
1$ is the smallest number with this property. Clearly, $X_0
\rightarrow X_1 \rightarrow \ldots \rightarrow X_r$ form a
sectional path of irreducible homomorphisms between modules in
$\ind A$ which starts in a direct summand $X_0$ of $\rad P/ \soc
(P)$. Then, for $l$ denoting the length function on $\mod A$, we
have the inequalities
$$l(X_{i-1})+l(Y_i)\geq l(X_i)+l(Y_{i-1}),$$
for all $i \in \{1,\ldots ,r\}$, and the equality holds if the
number of indecomposable direct summands in the middle term of the
Auslander-Reiten sequence ending in $Y_i$ is two. Therefore,
$l(X_{i-1})-l(Y_{i-1})\geq l(X_i)-l(Y_i)$ for all $i \in
\{1,\ldots ,r\}$. Since $l(X_r)=l(Q)
>l(Q/\soc (Q))=l(Y_r)$ we conclude that $l(X_0)>l(Y_0)=P/\soc (P)$, a
contradiction because $X_0$ is a proper submodule of $P/\soc (P)$.
\end{proof}

For the class of selfinjective algebras of finite representation
type without short cycles in the module category we have the
following property of stable slices.

\begin{lem} \label{4.A}
Let $A$ be a selfinjective algebra which does not admit a short
cycle in $\mod A$. Then all stable slices of $\Gamma_A$ are double
$\tau_A$-rigid.
\end{lem}
\begin{proof}
Since $\mod A$ has no short cycle, $A$ is of finite representation
type \cite[Theorem]{HL}. Suppose, on the contrary, that there is a
stable slice $\Delta$ of $\Gamma_A$ which is not double
$\tau_A$-rigid. Without loss of generality we may assume
$\Hom_A(X, \tau_AY) \neq 0$ for some indecomposable modules $X$
and $Y$ from $\Delta$. By $\Delta_Y$ we denote the stable slice of
$\Gamma_A$ formed by all modules $M$ in $\ind A$ such that there
is a sectional path in $\Gamma_A^s$ from $Y$ to $M$. Since $X$,
$Y$ belong to the same stable slice any nonzero homomorphism $f: X
\rightarrow \tau_AY$ factors through a direct sum $\oplus_{i=1}^m
Z_i$ of modules $Z_i$ from $\Delta_Y$ for some $m \geq 1$. Hence,
there is a nonzero homomorphism $g: Z \rightarrow \tau_AY$, where
$Z=Z_i$ for some $i\in \{1, \ldots,m\}$. Moreover,
$\Hom_A(Y,Z)\neq 0$ because the composition of irreducible
homomorphisms corresponding to the sectional path in $\Gamma_A^s$
is nonzero \cite{BS}. Hence the indecomposable module $Z$ is the
middle of a short chain of the form $Y \rightarrow Z \rightarrow
\tau_AY$. Applying \cite[Theorem 1.6]{RSS}, we conclude that $Z$
lies on a short cycle in $\mod A$. This contradicts the
assumption.
\end{proof}

The following example shows that the converse does not hold in
general.

\begin{exa}
Let $A=KQ/I$ be the bound quiver algebra, where $Q$ is the quiver
 \xymatrixcolsep{0.3cm}
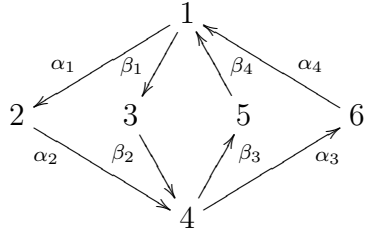
\[\xymatrix@!R@!C{ &&& 1 \ar[llld]_(0.65){\alpha_1} \ar[ld]_(0.6){\beta_1} &&& \\
2 \ar[rrrd]_(0.25){\alpha_2} && 3  \ar[rd]_(0.25){\beta_2} & & 5 \ar[lu]_(0.4){\beta_4} && 6 \ar[lllu]_(0.35){\alpha_4} \\
&&& 4 \ar[ru]_(0.75){\beta_3} \ar[rrru]_(0.75){\alpha_3} &&&}\]
and the ideal $I=\langle \alpha_1\alpha_2 -\beta_1\beta_2,
\alpha_3\alpha_4-\beta_3\beta_4, \alpha_2\beta_3, \beta_2\alpha_3,
\alpha_4\beta_1, \beta_4\alpha_1 \rangle$. Then $A$ is an orbit
algebra $\widehat{B}/G$, where $B$ is a tilted algebra of type
$\mathbb{A}_3$ and $G$ is an admissible group of automorphisms of
$\widehat{B}$ generated by $\nu^2_{\widehat{B}}$. The
Auslander-Reiten quiver $\Gamma_A$ of $A$ is of the form
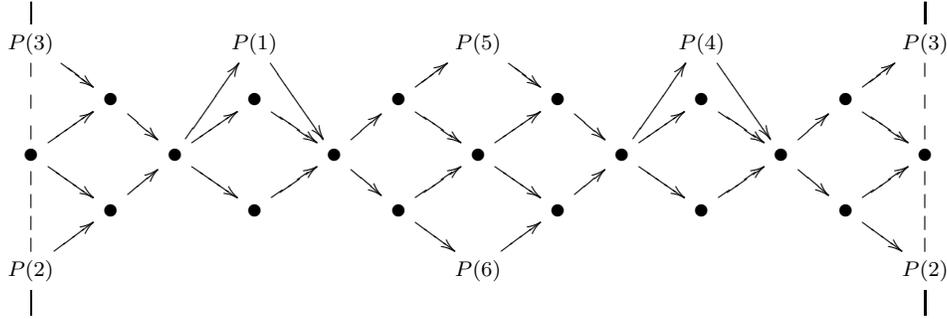
\xymatrixrowsep{0.3cm} \xymatrixcolsep{0.4cm}
\[\xymatrix{\ar@{--}[d]&&& &&& &&& &&& \\
\mathsmaller{P(3)} \ar@{-}[u] \ar[dr]\ar@{--}'[dd]'[dddd]&&&
\mathsmaller{P(1)} \ar[rdd]&&&
\mathsmaller{P(5)} \ar[rd] &&& \mathsmaller{P(4)}\ar[rdd] &&& \mathsmaller{P(3)}\ar@{--}'[dd]'[dddd] \ar@{-}[u]\\
& \bullet \ar[rd] & & \bullet\ar[rd] && \bullet \ar[rd] \ar[ru]&&
\bullet \ar[rd]&& \bullet \ar[rd] && \bullet\ar[rd] \ar[ru]
&\\
\bullet \ar[ru]\ar[rd] && \bullet \ar[ru]\ar[rd] \ar[ruu] &&
\bullet \ar[ru]\ar[rd]
&& \bullet \ar[ru]\ar[rd] && \bullet \ar[ruu] \ar[ru]\ar[rd] && \bullet \ar[ru]\ar[rd] && \bullet \\
& \bullet \ar[ru] & & \bullet\ar[ru] && \bullet \ar[ru] \ar[rd]&&
\bullet \ar[ru]&& \bullet \ar[ru] && \bullet\ar[ru]\ar[rd] &  \\
\mathsmaller{P(2)}\ar[ru]&&&&&&\mathsmaller{P(6)}\ar[ru] &&&&&&
\mathsmaller{P(2)}\\
\ar@{-}[u] &&& &&& &&& &&& \ar@{-}[u]}\] Let $\Delta$ be a stable
slice in $\Gamma_A$. We claim that $\Delta$ is double
$\tau_A$-rigid. Suppose there are indecomposable modules $X, Y \in
\Delta$ such that $\Hom_{A}(X, \tau_AY)\neq 0$. Since $A$ is of
finite representation type there exists a path of irreducible
homomorphisms \xymatrixcolsep{0.8cm}
\[\xymatrix{X=X_0 \ar[r]^(0.6){f_1} & X_1 \ar[r]^{f_2} & \ldots \ar[r]^(0.4){f_r} & X_r=\tau_A Y }\]
between nonprojective indecomposable modules $X_0, \ldots, X_r$
such that $f_r \ldots f_1 \neq 0$. Observe that $r\geq 8$. On the
other hand, $X_0, X_1, \ldots, X_r$ are of length at most $3$.
This contradicts Harada-Sai Lemma (see for example \cite[Lemma
III.2.1]{SY7}). Therefore, $\Hom_A(X, \tau_A Y)=0$ for any
indecomposable $X,Y \in \Delta$. Similarly, we prove that
$\Hom_A(\tau^-_{A}X, Y) =0$ for any indecomposable $X,Y \in
\Delta$. Therefore, $\Delta$ is  double $\tau_A$-rigid.
Nevertheless, $\xymatrix{P(1)\ar[r]^f & P(4) \ar[r]^g  & P(1)}$,
where  $\Im f=S(1)$ is the socle of $P(4)$ and $\Im g=S(4)$ is the
socle of $P(1)$, forms a short cycle in $\mod A$.
\end{exa}

The following consequence of Theorems 2.5, 2.6 and
\cite[Proposition 3.8  and Theorem 3.9]{SY8} will be crucial in
our proof of the main theorem.

\begin{thm} \label{4.4}
Let $A$ be a selfinjective algebra of finite representation type
such  that $\Gamma_A$ admits a semiregular double  $\tau_A$-rigid
stable slice $\Delta$. Moreover, let $M$ be the direct sum of the
indecomposable modules lying on $\Delta$, $I=r_A(M)$, and $B=A/I$.
Then the following statements hold.
\begin{enumerate}[\rm (i)]
\item $M$ is a tilting module in $\mod B$.
\item $H=\End_B(M)$ is a hereditary algebra of Dynkin type $\Delta$.
\item $T=D(M)$ is a tilting module in $\mod H$ with $\End_H(T)$ isomorphic to $B$.
\item $I$ is a deforming ideal of $A$ with $r_A(I)=eI$ for
 an idempotent $e$ of $A$, being a residual identity of $B=A/I$.
\item $A[I]$ is isomorphic to an orbit algebra $\widehat{B}/ (\psi \nu_{\widehat{B}})$ \
for a positive automorphism $\psi$ of $\widehat{B}$.
\item $A$ is socle equivalent to $A[I]$.
\end{enumerate}
\end{thm}

\section{Selfinjective Nakayama algebras.}

The aim of this section is to prove the implication
(i)$\Rightarrow$(ii) of the Theorem for selfinjective Nakayama
algebras.

It is well known that an algebra $B$ is a hereditary Nakayama
algebra if and only if $B$ is isomorphic to the algebra
$$T_n(F)= \left[\begin{matrix}
F & 0  &\cdots  & 0 \\
F & F  &\cdots  & 0  \\
\vdots &\vdots  &\ddots &\vdots \\
F & F  &\cdots & F
\end{matrix} \right]$$
of all lower triangular $n \times n$ matrices over a finite
dimensional division $K$-algebra $F$, for some positive natural
number $n$.

\begin{prop} \label{nowe4.1}
Let  $A$ be a selfinjective Nakayama algebra which does not admit
a short cycle in $\mod A$. Then $A$ is isomorphic to an orbit
algebra $\widehat{B}/ (\varphi \nu_{\widehat{B}}^2)$, where $B$ is
a hereditary Nakayama algebra and $\varphi$ is a strictly positive
automorphism of $\widehat{B}$.
\end{prop}
\begin{proof}
Let $P$ be an indecomposable projective module in $\mod A$. Since
$P$ is a uniserial module, the radical series
$$P \supset \rad P \supset \rad^2P \supset \cdots \supset \rad^nP \supset \rad^{n+1}P=0$$
of $P$ is its unique composition series, and hence $n+1=l(P)$ (see
\cite[Theorem I.10.1]{SY7}). Moreover, $n+1$ is the Loewy length
$ll(A)$ of $A$. We define $M_i=\rad^{n-i+1}P$ for $i \in
\{1,\ldots ,n\}$. It follows also from \cite[Proposition III.8.6
and Theorem III.8.7]{SY7} that there is in $\Gamma_A$ a sectional
path $\Delta$ of the form
$$M_1 \rightarrow M_2 \rightarrow \cdots \rightarrow M_{n-1} \rightarrow M_n$$
with $M_1=\soc (P)$ and $M_n=\rad P$. Moreover, $\Delta$ is a
stable slice of $\Gamma_A$. For each $i \in \{1,\ldots, n\}$, let
$P_i$ be the projective cover of $M_i$ in $\mod A$. We note that
there is a sectional path in $\Gamma_A$ of the form
$$P_i \rightarrow P_i/ \soc (P_i) \rightarrow \cdots \rightarrow M_i$$
and $M_i=P_i/ \rad^iP_i$, for any $i \in \{1,\ldots ,n\}$. We
claim that the projective modules $P_1,\ldots ,P_n$ are pairwise
different. Indeed, suppose that $P_i=P_j$ for some $i < j$ in
$\{1,\ldots ,n\}$. Clearly, we have in $\mod A$ a canonical proper
monomorphism $M_i \rightarrow M_j$. On the other hand,
$M_i=P_i/\rad^iP_i=P_j/\rad^iP_j$ and $M_j=P_j/\rad^jP_j$, and
consequently there is in $\mod A$ a proper epimorphism
$M_j\rightarrow M_i$. But then we have in $\mod A$ a short cycle
$M_i \rightarrow M_j \rightarrow M_i$, which contradicts the
assumption imposed on $A$. Observe now that we have the equalities
$$P_i/\soc (P_i)= \rad P_{i+1}= \tau_A(P_{i+1}/\soc (P_{i+1}))$$
for any $i \in \{1,\ldots ,n\}$, where $P_{n+1}=P$. We may choose
pairwise orthogonal primitive idempotents $e_j$, $j \in \{1,
\ldots ,m\}$, such that $1_A=e_1+\ldots +e_m$, $m \geq n$, and
$P_i=e_iA$ for any $i \in \{1,\ldots,n\}$. Let $M$ be the direct
sum of the modules $M_1,\ldots ,M_n$ lying on the stable section
$\Delta$, $I=r_A(M)$, and $B=A/I$. Observe that $e=e_1+ \ldots
+e_n$ is a residual identity of $B$. Since $A$ is a Nakayama
algebra, we conclude that $B$ is also a Nakayama algebra (see
\cite[Lemma I.10.2]{SY7}). Moreover, since $M$ is a faithful
module in $\mod B$, there is in $\mod B$ a monomorphism $B
\rightarrow M^r$ for some positive integer $r$ (see \cite[Lemma
II.5.5]{SY7}). Then we obtain that $M_1,\ldots ,M_n$ form a
complete set of pairwise nonisomorphic indecomposable projective
modules in $\mod B$, $B_B=M_1\oplus \ldots \oplus M_n$, and $B$ is
a hereditary Nakayama algebra of Loewy length $n$. Clearly, $M_n$
is the unique indecomposable projective-injective module in $\mod
B$, and the Auslander-Reiten quiver $\Gamma_B$ of $B$ is a full
valued translation subquiver of the Auslander-Reiten quiver
$\Gamma_A$ of $A$. Consider the indecomposable projective modules
$P_{n+1},\ldots ,P_{2n}$ in $\mod A$ such that $\rad
P_j=P_{j-1}/\soc (P_{j-1})$ for $j \in \{n+1,\ldots ,2n\}$.
Further, let $\nu$ be the Nakayama permutation of $A$, which is  a
permutation of $\{1,\ldots,m\}$. Then it follows from the above
discussion that $j=\nu(j-n)$ for any $j \in \{n+1,\ldots,2n\}$.
Clearly, $P_{n+1},\ldots,P_{2n}$ are pairwise different projective
modules in $\Gamma_A$, because the projective modules
$P_1,\ldots,P_n$ are pairwise different. We also observe that
$P_i\neq P_j$ in $\Gamma_A$ for any $i \in \{1,\ldots,n\}$ and $j
\in \{n+1,\ldots,2n\}$. Indeed, suppose that $P_i=P_j$ for some $i
\in \{1,\ldots,n\}$ and $j \in \{n+1, \ldots,2n\}$. Then there is
a short cycle in $\mod A$ of the form
$$\xymatrix{P_i \ar[r]^f & M_n \ar[r]^(.4)g  & P_j=P_i}$$
because the factor module $M_i$ of $P_i$ is a submodule of $M_n$
and a factor module of $M_n$ is a submodule of $P_j$,
contradicting our assumption on $A$. In particular, we conclude
that $m \geq 2n$.

We will prove now that $I$ is a deforming ideal of $A$ with
$l_A(I)=Ie$ and $r_A(I)=eI$. We first note that the valued quiver
$Q_B$ of $B$ is acyclic, because $B$ is a hereditary Nakayama
algebra with $Q_B$ of the form
$$\xymatrix{1  & 2\ar[l]  & \ldots \ar[l]  & n-1 \ar[l] & n \ar[l]}.$$
Then, in order to show that $I$ is a deforming ideal of $A$, it is
enough to show that $l_A(I)=Ie$, by Proposition \ref{2.2}. We
denote by $J$ the trace ideal of $M$ in $A$, that is, the ideal in
$A$ generated by the images of all homomorphisms from $M$ to $A$
in $\mod A$. Observe that $J$ is a right $B$-module, and hence
$Je=J$ and $JI=0$. Since $\mod A$ is without short cycles and
$\Delta$ is a stable slice of $\Gamma_A$, it follows from Lemma
\ref{4.A} that $\Delta$ is double $\tau_A$-rigid. Then, applying
\cite[Lemmas 3.10, 3.11 and 3.12]{SY8}, we obtain that $J
\subseteq I$, $l_A(I)=J$, and $eIe=eJe$. We claim that
$(1-e)Ie=(1-e)Je$. Clearly, $(1-e)I=(1-e)A$ because $1-e \in I$ by
the definition of $e$. Further, there is a canonical isomorphism
of right $eAe$-modules $(1-e)Ae\cong \Hom_A(eA,(1-e)A)$ in $\mod
A$ (see \cite[Lemma I.8.7]{SY7}). Let $f: eA \rightarrow (1-e)A$
be a nonzero homomorphism in $\mod A$. Since $eA =P_1 \oplus
\cdots \oplus P_n$, it follows from the definition of $\Delta$
that there exists a positive integer $s$ and homomorphisms $g: eA
\rightarrow M^s$, $h: M^s \rightarrow (1-e)A$ such that $f=hg$.
But then $\Im f \subseteq \Im h \subseteq (1-e)J$. This shows that
$(1-e)Ie=(1-e)Ae \subseteq (1-e)Je$. Then we obtain that $Ie = eIe
\oplus (1-e)Ie \subseteq eJe \oplus (1-e)Je =J$. Since $J
\subseteq I$ we have also $J=Je \subseteq Ie$. Summing up, we
conclude that $Ie=J=l_A(I)$. In particular, we get $IeI=0$. Then,
applying Proposition \ref{2.2}, we obtain that $I$ is a deforming
ideal of $A$ with $l_A(I)=Ie$ and $r_A(I)=eI$.

It follows from Theorem \ref{2.5} that $A[I]$ is isomorphic to an
orbit algebra $\widehat{B}/(\psi \nu_{\widehat{B}})$ for some
positive automorphism $\psi$ of $\widehat{B}$. Moreover, by
Theorem \ref{2.4}, $A$ is socle equivalent to $A[I]$ and both
algebras have the same Nakayama permutation. Observe also that, if
$e_i$ is a primitive summand of $e$, then $i \in \{1,\ldots,n\}$
and $e_{\nu(i)}=e_{n+i}$, and consequently $e_i \neq e_{\nu(i)}$.
Then it follows from Theorem \ref{2.6} that the algebras $A$ and
$A[I]$ are isomorphic. In particular, we conclude that $A$ is
isomorphic to the orbit algebra
$\widehat{B}/(\psi\nu_{\widehat{B}})$. Finally, we claim that
$\psi=\varphi \nu_{\widehat{B}}$ for a strictly positive
automorphism $\varphi$ of $\widehat{B}$. Since $P_1, \ldots, P_n,
P_{n+1}, \ldots, P_{2n}$  are pairwise different projective
modules in $\Gamma_A$, and  $n+i=\nu(i)$ for any $i \in
\{1,\ldots,n\}$, we conclude that $\psi=\varphi\nu_{\widehat{B}}$
for a positive automorphism $\varphi$ of $\widehat{B}$. Suppose
that $\varphi$ is a rigid automorphism of $\widehat{B}$. Then
$m=2n$ and $\nu^2$ is the identity permutation of
$\{1,\ldots,m\}$. Then we have in $\mod A$ a short cycle of the
form
$$\xymatrix{P_1 \ar[r]^u & P_{n+1} \ar[r]^v  & P_1}$$
with $\Im u$ and $\Im v$ being simple modules, which contradicts
the assumption imposed on $A$. Therefore, we conclude that $A$ is
isomorphic to an orbit algebra $\widehat{B}/(\varphi
\nu^2_{\widehat{B}})$ for a strictly positive automorphism
$\varphi$ of $\widehat{B}$.
\end{proof}

\section{Selfinjective algebras of Dynkin type.}

Let $B$ be a triangular algebra (i.e. the quiver $Q_{B}$ is
acyclic) and $e_{1}, \ldots ,e_{n}$ be pairwise orthogonal
primitive idempotents of $B$ with $1_{B}=e_{1}+\cdots+e_{n}$. We
identify $B$ with the full subcategory $B_{0}$ of the repetitive
category $\widehat B$ given by the objects $e_{0,j}$, $1\leq j\leq
n$. For a sink $i$ of $Q_{B}$, the {\it reflection} $S^{+}_{i}B$
of $B$ at $i$ is the full subcategory of $\widehat B$ given by the
objects
\begin{displaymath}
e_{0,j},\quad 1\leq j\leq n, \quad j\neq i, \quad \textrm{and} \quad e_{1,i}=\nu_{\widehat B}(e_{0,i}).
\end{displaymath}
Then the quiver $Q_{S^{+}_{i}B}$ of $S^{+}_{i}B$ is the reflection
$\sigma^{+}_{i}Q_{B}$ of $Q_{B}$ at $i$ (see \cite{HW}). Observe
that $\widehat B=\widehat{S^{+}_{i}B}$. By a \emph{reflection
sequences of sinks} of $Q_{B}$ we mean a sequence $i_{1}, \ldots
,i_{t}$ of vertices of $Q_{B}$ such that $i_{s}$ is a sink of
$\sigma^{+}_{i_{s-1}} \ldots \sigma^{+}_{i_{1}}Q_{B}$ for all $s$
in $\{1, \ldots ,t\}$. Moreover, for a sink $i$ of $Q_{B}$, we
denote by $T^{+}_{i}B$ the full subcategory of $\widehat B$ given
by the objects
\begin{displaymath}
e_{0,j},\quad 1\leq j\leq n, \quad \textrm{and} \quad e_{1,i}=\nu_{\widehat B}(e_{0,i}).
\end{displaymath}
Observe that $T^{+}_{i}B$ is the one-point extension $B[I_{B}(i)]$
of $B$ by the indecomposable injective $B$-module $I_{B}(i)$ at
the vertex $i$.  Recall that by a finite dimensional $\widehat
B$-module we mean a contravariant $K$-linear functor $M$ from
$\widehat B$ to the category of $K$-vector spaces such that
$\sum_{x\in{ob\widehat B}}\dim _{K}M(x)$ is finite. We denote by
$\mod \widehat B$ the category of all finite dimensional $\widehat
B$-modules. Finally, for a module $M$ in mod$\widehat B$, we
denote by $\supp (M)$ the full subcategory of $\widehat B$ formed
by all objects $x$ with $M(x)\neq 0$, and call it the
\emph{support of M}.

The following consequences of results proved in \cite{Ho},
\cite{HW} describe the supports of  finite dimensional
indecomposable modules over the repetitive categories $\widehat B$
of tilted algebras $B$ of Dynkin type.

\begin{thm} \label{3.1}
Let $B$ be a tilted algebra of Dynkin type $\Delta$ and $n$ the
rank of $K_{0}(B)$. Then there exists a reflection sequence
$i_{1},\ldots ,i_{n}$ of sinks of $Q_{B}$  such that the following
statements hold.
\begin{enumerate}[\upshape (i)]
\item $S^{+}_{i_{n}}\ldots S^{+}_{i_{1}}B=\nu_{\widehat B}(B)$.
\item For each $r \in \{1,\ldots,n\}$, $S^+_{i_r}\ldots S^+_{i_1}B$ is a
tilted algebra of type $\Delta$.
\item For every indecomposable nonprojective module $M$ in $\mod \widehat B$, $\supp (M)$
is contained in one of the full subcategories of $\widehat B$
\begin{displaymath}
\nu^{m}_{\widehat B}(S^{+}_{i_{r}}\ldots S^{+}_{i_{1}}B),\quad r\in{\{1, \ldots ,n \}},\quad m\in{\mathbb{Z}}.
\end{displaymath}
\item For every indecomposable projective module $P$ in $\mod \widehat B$, $\supp (P)$
 is contained in one of the full subcategories of $\widehat B$
\begin{displaymath}
\nu^{m}_{\widehat B}(T^{+}_{i_{r}}S^{+}_{i_{r-1}}\ldots
S^{+}_{i_{1}}B),\quad r\in{\{1, \ldots ,n \}},\quad
m\in{\mathbb{Z}}.
\end{displaymath}
\end{enumerate}
\end{thm}

Moreover, invoking again \cite{Ho} and \cite{HW}, we have the
below proposition.

\begin{prop}\label{3.2}
Let $B$  be a tilted algebra of Dynkin type and $G$ be an
admissible group of automorphisms  of $\widehat{B}$. Then $G$ is
an infinite cyclic group generated by a strictly positive
automorphism of $\widehat{B}$.
\end{prop}

It follows from Theorem~\ref{3.1} that the repetitive category
$\widehat{B}$ of  a tilted algebra $B$ of Dynkin type is a {\it
locally representation-finite} category \cite{G2}, that is, for
any object $x \in \widehat{B}$ the number of indecomposable
modules $N \in \mod \widehat{B}$ satisfying $N(x) \neq 0$ is
finite. Then we obtain the following consequence
of~\cite[Theorem~3.6]{G2}.

\begin{thm}\label{3.3}
Let $B$ be a tilted algebra of Dynkin type, $G$ an admissible
group of automorphisms of  $\widehat{B}$, and $A=\widehat{B}/G$
the associated selfinjective orbit algebra. Then the following
statements hold.
\begin{enumerate}[\upshape (i)]
\item The push-down functor $F_{\lambda}\colon\mod \widehat{B}\to\mod A$,
associated with the Galois covering
$F\colon\widehat{B}\to\widehat{B}/G=A$, is  dense and preserves
the almost split sequences.
\item The Auslander-Reiten quiver $\Gamma_{A}$ of $A$ is the orbit quiver
$\Gamma_{\widehat{B}}/G$ of $\Gamma_{\widehat{B}}$ with respect to
the induced action of $G$ on $\Gamma_{\widehat{B}}$.
\item $A$ is of finite representation type.
\end{enumerate}
\end{thm}

For a Dynkin quiver $\Delta$ of type $\mathbb{A}_n$ $(n\geq 1)$,
$\mathbb{B}_n$ $(n \geq 2)$, $\mathbb{C}_n$ $(n \geq 3)$,
$\mathbb{D}_n$ $(n \geq 4)$, $\mathbb{E}_6$, $\mathbb{E}_7$,
$\mathbb{E}_8$, $\mathbb{F}_4$ or $\mathbb{G}_2$, by a {\it
selfinjective algebra of type} $\Delta$ we mean an orbit algebra
$\widehat{B}/G$, where $B$ is a tilted algebra of type $\Delta$
and $G$ is an admissible group of automorphisms of $\widehat{B}$.

\begin{prop}\label{3.A}
Let $A$ be a selfinjective algebra of Dynkin type $\Delta
\in\{\mathbb{B}_n, \mathbb{C}_n, \mathbb{F}_4,\mathbb{G}_2\}$.
Then $A$ is isomorphic to the $r$-fold trivial extension algebra
$T(B)^{(r)}$, where $B$ is a tilted algebra of type $\Delta$ and
$r$ is a positive integer.
\end{prop}

\begin{proof}
For $\Delta$ of type $\mathbb{B}_n$ and $\mathbb{C}_n$ this
follows from \cite[Theorems 1.2, 1.3 and 4.1]{YX}. We note that
the proofs of these results presented in \cite{YX} apply
essentially combinatorial characterizations of finite
Auslander-Reiten quivers of algebras, established in \cite{IT},
and the classification of selfinjective configurations of the
stable translation quivers $\mathbb{ZD}_{n+1}$ and
$\mathbb{ZA}_{2n-1}$ given in \cite{Rd2}, \cite{Rd3}. For $\Delta=
\mathbb{F}_4$ the claim follows by repeating arguments applied in
\cite{YX} for types $\mathbb{B}_n$ and $\mathbb{C}_n$ and the
classification of selfinjective configurations of the stable
translation quiver $\mathbb{ZE}_6$ established in \cite{BLR}.
Similarly, for $\Delta=\mathbb{G}_2$, the proof reduces to the
classification of selfinjective configurations of the stable
translation quiver $\mathbb{ZD}_4$ (see \cite[(7.6)]{BLR}).
\end{proof}

Let  $B$ be a tilted algebra of Dynkin type $\Delta$ and $n$ the
rank of $K_0(B)$. Following \cite{S}, $B$ is said to be {\it
exceptional} if there exists a reflection sequence $i_1, \ldots,
i_t$ of sinks in $Q_B$ with $t< n$ and $S^+_{i_t} \ldots
S^+_{i_1}B$ isomorphic to $B$. Equivalently, $B$ is exceptional if
and only if there is an automorphism $\varphi$ of $\widehat{B}$
with $\varphi^m=\nu_{\widehat{B}}$ for some $m\geq 2$ (see
\cite[Proposition 3.9]{S}). We have the following consequence of
\cite[Proposition 1.4 and 1.5]{BLR}, \cite[Proposition 3.3]{Rd2},
\cite[Theorem]{Rd3} and Proposition~\ref{3.A}).
\begin{prop}
Let $B$ be an exceptional tilted algebra of Dynkin type $\Delta$.
Then either $\Delta=\mathbb{A}_n$ for some $n\geq 2$ or
$\Delta=\mathbb{D}_{3m}$ for some $m \geq 2$.
\end{prop}

We introduce now exceptional tilted algebras of Dynkin type which
play a prominent role in the description of selfinjective algebras
of Dynkin types $\mathbb{A}_n$ $(n \geq 1)$, $\mathbb{D}_n$
$(n\geq 4)$, $\mathbb{E}_6$, $\mathbb{E}_7$, $\mathbb{E}_8$ (see
\cite[Section 3]{S} for details). Let $F$ be a division algebra
over a field $K$.

Let $T^m_S$ be a Brauer tree with $n$ edges and an exceptional
vertex $S$ of multiplicity $m \geq 2$. Then the Brauer tree
algebra $A(F, T^m_S)$ of $T^m_S$ over $F$ is a symmetric algebra
of Dynkin type $\mathbb{A}_n$, isomorphic to the orbit algebra
$\widehat{B(F, T^m_S)}/ (\varphi)$, for an exceptional  tilted
algebra $B(F, T^m_S)$ of type $\mathbb{A}_n$ and an automorphism
$\varphi$ of $\widehat{B(F, T^m_S)}$ with
$\varphi^m=\nu_{\widehat{B(F,T^m_S)}}$.

Let $T^2_S$ be a Brauer tree with $m \geq 2$ edges and an extreme
exceptional vertex $S$ of multiplicity $2$. Then the modified
Brauer tree algebra $D(F, T^2_S)$ of $T^2_S$ over $F$ (in the
sense of \cite{Rd3}, \cite{W2}) is  a symmetric algebra of Dynkin
type $\mathbb{D}_{3m}$, isomorphic to the orbit algebra
$\widehat{B^*(F, T^2_S)}/(\varphi)$, for an exceptional tilted
algebra $B^*(F, T^2_S)$ of type $\mathbb{D}_{3m}$ and an
automorphism $\varphi$ of $\widehat{B^*(F, T^2_S)}$ with
$\varphi^3=\nu_{\widehat{B^*(F, T^2_S)}}$.

Then we have the following general version of \cite[Theorem
3.10]{S}.
\begin{prop} \label{3.6}
Let $A$ be a selfinjective algebra of Dynkin type\\ $\Delta \in
\{\mathbb{A}_n, \mathbb{D}_n, \mathbb{E}_6, \mathbb{E}_7,
\mathbb{E}_8\}$. Then $A$ is isomorphic to an orbit algebra of one
of the forms:
\begin{enumerate}[\upshape (1)]
\item $\widehat{B}/(\nu^r_{\widehat{B}})$, $r\geq 1$, $B$ a tilted
algebra of type $\Delta \in \{\mathbb{A}_n, \mathbb{D}_n,
\mathbb{E}_6, \mathbb{E}_7, \mathbb{E}_8\}$;
\item $\widehat{B}/(\sigma\nu^r_{\widehat{B}})$, $r\geq 1$, $B$ a tilted
algebra of type $\Delta \in \{\mathbb{A}_{2p+1}, \mathbb{D}_n,
\mathbb{E}_6\}$, $\sigma$ an automorphism of $\widehat{B}$ of
order $2$;
\item $\widehat{B}/(\sigma\nu^r_{\widehat{B}})$, $r\geq 1$, $B$ a
hereditary algebra with $Q_B$ of the form \xymatrixrowsep{0.6cm}
\xymatrixcolsep{0.4cm}
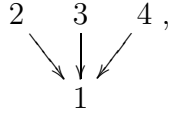
$$\xymatrix{2 \ar[rd] & 3 \ar[d] & 4 \ar[ld]\\
& 1 &},$$ $\sigma$ an automorphism of $\widehat{B}$ of order $3$;
\item $\widehat{B}/(\varrho^r)$, $r\geq 1$, $B=B(F, T^m_S)$ the
tilted algebra of type $\mathbb{A}_n$ (introduced above),
$\varrho$ an automorphism of $\widehat{B}$ with
$\varrho^m=\nu_{\widehat{B}}$;
\item $\widehat{B}/(\varrho^r)$, $r\geq 1$, $B=B^*(F, T^2_S)$ the
tilted algebra of type $\mathbb{D}_{3m}$ (introduced above),
$\varrho$ an automorphism of $\widehat{B}$ with
$\varrho^3=\nu_{\widehat{B}}$.
\end{enumerate}
\end{prop}
We are now in the position to prove the following proposition,
forming an essential step in the proof of the main theorem.
\begin{thm}\label{3.4}
Let $B$ be a tilted algebra of Dynkin type $\Delta$, $G$ an
admissible group of automorphisms of $\widehat{B}$, and
$A=\widehat{B}/ G$ the associated orbit algebra. The following
statements are equivalent.
\begin{enumerate}[\upshape (i)]
\item $\mod A$ has no short cycles.
\item $G= (\varphi \nu^2_{\widehat{B}})$ for a strictly positive  automorphism
$\varphi$ of $\widehat{B}$.
\end{enumerate}
\end{thm}

\begin{proof}
Let $n$ be the rank of $K_0(B)$ and $i_1, \ldots, i_n$ a
reflection sequence of sinks in $Q_B$ satisfying the statements of
Theorem \ref{3.1}. Applying Proposition \ref{3.2} we may choose a
strictly positive automorphism $g$ of $\widehat{B}$ generating
$G$. Moreover, let $F_{\lambda}\colon\mod \widehat{B}\to\mod A$ be
the push-down functor associated with the Galois covering
$F:\widehat{B} \to \widehat{B}/G=A$. Further, let $e_1, \ldots,
e_n$ be  pairwise orthogonal primitive idempotents of $B$ such
that $1_B=e_1+ \ldots e_n$. Then $e_{m,i}$, $(m,i) \in
\mathbb{Z}\times \{1, \ldots, n\}$, form the set of objects of
$\widehat{B}$. For each $(m,i) \in \mathbb{Z}\times \{1, \ldots,
n\}$, we denote by $P(m,i)$ the indecomposable projective module
$\Hom_{\widehat{B}}(-, e_{m,i})$ in $\mod \widehat{B}$.

We prove first that (ii) implies (i). Assume that
$g=\varphi\nu^2_{\widehat{B}}$ for a strictly positive
automorphism $\varphi$ of $\widehat{B}$. Suppose that there is a
short cycle
$$\xymatrix{M \ar[r]^u & N \ar[r]^v  & M}$$
in $\mod A$. It follows from Theorem \ref{3.3} that there exist
indecomposable modules $X$ and $Y$ in $\mod \widehat{B}$ such that
$M=F_{\lambda}(X)$ and $N=F_{\lambda}(Y)$. Further, since
$F_{\lambda}: \mod \widehat{B} \to \mod A$ is a Galois covering of
module categories (see \cite[Theorem 3.6]{G2}) the functor
$F_{\lambda}$ induces a $K$-linear isomorphism
$$\xymatrix{\bigoplus_{p\in \mathbb{Z}} \Hom_{\widehat{B}}(X,
g^pY)\ar[r]^(.49){\sim}& \Hom_A(F_{\lambda}(X),
F_{\lambda}(Y))}.$$ Hence, there is a nonzero nonisomorphism $f: X
\to g^p Y$ for some $p \in \mathbb{Z}$. Let $Z=g^pY$. Clearly, we
have $F_{\lambda}(Z)=F_{\lambda}(g^pY)=F_{\lambda}(Y)=N$. The
functor $F_{\lambda}$ induces also a $K$-linear isomorphism
$$\xymatrix{\bigoplus_{s\in \mathbb{Z}} \Hom_{\widehat{B}}(Z,
g^sX)\ar[r]^(.49){\sim}& \Hom_A(F_{\lambda}(Z),
F_{\lambda}(X))},$$ and hence there is a nonzero nonisomorphism
$h: Z \colon\rightarrow g^sX$ for some $s \in \mathbb{Z}$. We note
that $\Gamma_{\widehat{B}}$ is an acyclic quiver whose stable part
$\Gamma^s_{\widehat{B}}$ is isomorphic to $\mathbb{Z}\Delta$.
Moreover, since $\widehat{B}$ is locally representation-finite,
every nonzero nonisomorphism between indecomposable modules in
$\mod \widehat{B}$ is a sum of compositions of irreducible
homomorphisms between indecomposable modules in $\mod
\widehat{B}$. Then we conclude that $s \geq 1$. We have two cases
to consider.

Assume first that $X$ is not projective. It follows from Theorem
\ref{3.1} that $\supp (X)$ is contained in the full subcategory of
$\widehat{B}$ of the form $\nu_{\widehat{B}}^m(S^+_{i_r}\ldots
S^+_{i_1}B)$, for some $r \in \{1, \ldots, n\}$ and $m \in
\mathbb{Z}$. Since $B'=\nu^m_{\widehat{B}}(S^+_{i_r}\ldots
S^+_{i_1}B)$ is a tilted algebra of Dynkin type $\Delta$ with
$\widehat{B'}=\widehat{B}$, we may assume that $\supp(X)$ is
contained in $B$. We note that $\Hom_{\widehat{B}}(X,Z)\neq 0$
implies that the supports $\supp (X)$ and $\supp(Z)$ are not
disjoint. Then, applying Theorem \ref{3.1} again, we conclude that
$\supp(Z)$ is contained in a full subcategory of $\widehat{B}$ of
the form $T^+_{i_k}S^+_{i_{k-1}}\ldots S^+_{i_1}B$ for some $k \in
\{1, \ldots, n\}$. On the other hand, since
$g=\varphi\nu^2_{\widehat{B}}$ with $\varphi$ strictly positive
automorphism of $\widehat{B}$ and $s\geq 1$, $\supp(g^sX)$ is
contained in the full subcategory of $\widehat{B}$ given by  the
objects $e_{m,i}$ for all $(m,i) \in \mathbb{Z}\times
\{1,\ldots,n\}$ with $m\geq 2$. Hence, the supports $\supp(Z)$ and
$\supp (g^sX)$ are disjoint, and consequently
$\Hom_{\widehat{B}}(Z,g^sX) =0$, a contradiction.

Finally, assume that $X$ is projective. Then, by Theorem
\ref{3.1}, $\supp(X)$ is contained in a full subcategory of
$\widehat{B}$ of the form
$\nu^m_{\widehat{B}}(T^+_{i_r}S^+_{i_{r-1}}\ldots S^+_{i_1}B)$ for
some $r \in \{1, \ldots, n\}$ and $m \in \mathbb{Z}$. We may
assume that $\supp(X)$ is contained in $T^+_{i_1}B$ and hence
$X=P(1,i_1)$. Since there is a nonzero nonisomorphism $f:X
\rightarrow Z$, we conclude that $\Hom_{\widehat{B}}(X/\soc(X),
Z)\neq 0$. Observe also that $\supp(X/\soc(X))$ is contained in
$S^+_{i_1}B$, which is the full subcategory of $\widehat{B}$ with
the objects $e_{1,i_1}$ and $e_{0,i}$, for $i \in \{1, \ldots,
n\}\backslash\{i_1\}$. Then we obtain that $\supp(Z)$ is contained
in the full subcategory of $\widehat{B}$ given by the objects of
$S^+_{i_1}B$ and $\nu_{\widehat{B}}(S^+_{i_1}B)$, so by the
objects $e_{1,i_1}, e_{2,i_1}$, and $e_{0,i}, e_{1,i}$ for $i \in
\{1, \ldots, n\}\backslash \{i_1\}$. Observe also that
$\supp(g^sX)=\supp(g^sP(1,i_1))$ is contained in the full
subcategory of $\widehat{B}$ given by the objects $g^s(e_{1,i_1})$
and $g^s(e_{0,i})$ with $i \in \{1, \ldots, n\}$. Clearly,
$g^s(e_{1,i_1})=e_{k,j}$ for some $k\geq 3$ and $j \in \{1,\ldots,
n \}$, because $g=\varphi\nu^2_{\widehat{B}}$ and $s\geq 1$.
Similarly, for any $i \in \{1, \ldots, n\}$, we have
$g^s(e_{0,i})=e_{k,j(i)}$, for some $k\geq 2$ and $j(i)\in \{1,
\ldots, n \}$. Since $\Hom_{\widehat{B}}(Z, g^sX)\neq 0$, we
conclude that the supports $\supp(Z)$ and $\supp(g^sX)$ have a
common object. Then $e_{2,i_1}$ is the unique common object of
$\supp(Z)$ and $\supp(g^sX)$. But then $s=1$ and
$g(e_{0,i})=e_{2,i_1}$ for some $i \in \{1, \ldots, n\}\backslash
\{i_1\} $. Hence
$\varphi(e_{2,i})=\varphi(\nu^2_{\widehat{B}}(e_{0,i}))=g(e_{0,i})=e_{2,
i_1 }$. Applying $\nu^{-1}_{\widehat{B}}$ to this equality we
obtain that $\varphi(e_{1,i})=e_{1,i_1}$, and consequently
$\varphi P(1,i)=P(1,i_1)$. Since $i=i_t$ for some $t \in \{2,
\ldots, n\}$, if $\varphi P(1, i_1)=P(1,j)$ for some $j \in \{1,
\ldots, n \}$, then $j=i_s$ for some $s \in \{2, \ldots, n\}$ and
$s\neq t$, because $\varphi$ induces a strictly positive
automorphism of the acyclic Auslander-Reiten quiver
$\Gamma_{\widehat{B}}$. Then $gX=gP(1,
i_1)=\varphi\nu_{\widehat{B}}^2P(1,i_1)=P(3,j)$ and $\supp(gX)\cap
\supp(Z) = \supp(P(3,j))\cap \supp (Z)$ contains $e_{2,j}$ because
$\Hom_{\widehat{B}}(Z, gX) \neq 0$. Therefore $e_{2,j}=e_{2,i_1}$,
a contradiction. Similarly, if $\varphi P(1, i_1)=P(k,j)$ for some
$k \geq 2$ and $j\in \{1,\ldots,n\}$, then $gX
=gP(1,i_1)=\varphi\nu_{\widehat{B}}^2P(1,i_1)=P(k,j)$ for some $k
\geq 4$. Hence, the supports $\supp(Z)$ and $\supp (gX)$ are
disjoint, and consequently $\Hom_{\widehat{B}}(Z,gX) =0$, a
contradiction. Summing up, we have proved that (ii) implies (i).

We shall now prove that (i) implies (ii). Assume that $g$ is not
of the form $\varphi\nu^2_{\widehat{B}}$ for a strictly positive
automorphism $\varphi$ of $\widehat{B}$. We will show that $\mod
A$ contains a short cycle. Recall that $g$ is a strictly positive
automorphism of $\widehat{B}$. We have few cases to consider,
taking into account Proposition \ref{3.6}.

\textbf{(a)} Assume that $g=\sigma\nu_{\widehat{B}}$ or $g=\sigma
\nu^2_{\widehat{B}}$ for a rigid automorphism $\sigma$ of
$\widehat{B}$, that is, $A=\widehat{B}/G$ is as in Proposition
\ref{3.A} with $r \in \{1,2\}$, or is one of the forms (1), (2),
(3) with $r \in \{1,2\}$, described in Proposition \ref{3.6}. Then
it follows from the classification of selfinjective configurations
of stable quivers $\mathbb{ZA}_{2p+1}$, $\mathbb{ZD}_n$ and
$\mathbb{ZE}_6$ given in \cite{BLR}, \cite{Rd1}, \cite{Rd2},
\cite{Rd3} (see also \cite{MP} for a general result) that there is
$(m,i) \in \mathbb{Z}\times \{1, \ldots, n\}$ such that
$\sigma(e_{m,i})=e_{m,i}$, and then $\sigma(e_{k,i})=e_{k,i}$ for
all $k \in \mathbb{Z}$. Hence, we obtain that $\sigma
P(k,i)=P(k,i)$ for all $k \in \mathbb{Z}$. But then
$\sigma\nu_{\widehat{B}}P(0,i)=\sigma P(1,i)=P(1,i)$ and
$\sigma\nu^2_{\widehat{B}}P(0,i)=\sigma P(2,i)=P(2,i)$. Observe
that we have in $\mod \widehat{B}$ a sequence of nonzero
nonisomorphisms
\[\xymatrix{P(0,i) \ar[r]^f & P(1,i)\ar[r]^h& P(2,i)}\]
with $\Im f =\soc (P(1,i))$ and $\Im h =\soc (P(2,i))$. Then we
obtain a short cycle in $\mod A$ of the form
\[\xymatrix{F_{\lambda}(P(0,i)) \ar[r]^{F_{\lambda}(f)} & F_{\lambda}(P(1,i))\ar[r]^{F_{\lambda}(h)}& F_{\lambda}(P(0,i))},\]
where $F_{\lambda}(P(2,i))=F_{\lambda}(P(0,i))$, because if
$g=\sigma \nu^2_{\widehat{B}}$, then
$gP(0,i)=\sigma\nu_{\widehat{B}}^2P(0,i)=P(2,i)$. We also note
that if $g=\sigma\nu_{\widehat{B}}$, then
$F_{\lambda}(P(1,i))=F_{\lambda}(P(0,i))$.

\textbf{(b)} Assume that $B$ is an exceptional tilted algebra of
type $\mathbb{A}_n$, and $\varrho$ is an automorphism of
$\widehat{B}$ with $\varrho^m=\nu_{\widehat{B}}$ and $m \geq 2$.
Then $g=\varrho^p$ for some $p \in \{1, \ldots, 2m \}$, by the
assumption imposed on $g$. Applying Proposition \ref{3.6}, we
conclude that $\widehat{B}\cong \widehat{B(F, T^m_S)}$ for the
exceptional tilted algebra $B(F, T^m_S)$ associated with a
division algebra $F$ and a Brauer tree $T^m_S$ with $n$ edges and
an exceptional vertex $S$ of multiplicity $m$. We may assume that
$B=B(F, T^m_S)$. Then $\widehat{B}/(\varrho)$ is the Brauer tree
algebra $A(F, T^m_S)$ associated with $F$ and $T^m_S$. Consider
the orbit algebra $A_k=\widehat{B}/(\varrho^k)$, for $k \in \{1,
\ldots, 2m\}$. Then $A_k=FQ_{A_k}/I_{A_k}$, where $Q_{A_k}$ is the
quiver of $A_k$ and $I_{A_k}$ is an ideal of the path algebra
$FQ_{A_k}$ of $Q_{A_k}$ over $F$. Since $A_1=A(F, T^m_S)$, the
quiver $Q_{A_1}$ contains a cycle $C_1$ of length $l$, with $l$
being the number of edges in $T^m_S$ connected to the exceptional
vertex $S$, and the composition of any $lm$ consecutive arrows in
$C_1$ does not belong to $I_{A_1}$. Hence, for any $k \in \{1,
\ldots, 2m\}$, the quiver $Q_{A_k}$ contains a cycle $C_k$ of
length $kl$ such that the composition of any $lm$ consecutive
arrows in $C_k$ does not belong to $I_{A_k}$. But then the cycle
$C_k$ contains two vertices $x$ and $y$ such that the shortest
paths $u$ from $x$ to $y$ and $v$ from $y$ to $x$ in $C_k$ do not
belong to $I_{A_k}$, and then the indecomposable projective
modules $P_{A_k}(x)$ and $P_{A_k}(y)$ in $\mod A_k$ at the
vertices $x$ and $y$ lie on a short cycle
\[P_{A_k}(x) \rightarrow P_{A_k}(y) \rightarrow P_{A_k}(x).\]
In particular, since $A=\widehat{B}/G=\widehat{B}/(\varrho^p)$,
$\mod A$ admits a short cycle of indecomposable projective
modules.

\textbf{(c)} Assume that $B$ is an exceptional tilted algebra of
type $\mathbb{D}_{3m}$, and $\varrho$ is an automorphism of
$\widehat{B}$ with $\varrho^3=\nu_{\widehat{B}}$. Then
$A=\widehat{B}/(\varrho^p)$ for some $p \in \{1, \ldots, 6\}$, by
the assumption imposed on $g$. Applying Theorem \ref{3.6}, we
conclude that $\widehat{B}=B^*(F,T^2_S)$ for the exceptional
tilted algebra $B^*(F,T^2_S)$ associated with a division algebra
$F$ and a Brauer tree $T^2_S$ with $m \geq 2$ edges and an extreme
exceptional vertex $S$ of multiplicity $2$. We may assume that
$B=B^*(F, T^2_S)$. Then $\widehat{B}/(\varrho)$ is isomorphic to
the modified Brauer algebra $D(F, T^2_S)$ associated with $F$ and
$T^2_S$. Consider the orbit algebra $D_k=\widehat{B}/(\varrho^k)$,
for $k \in \{1, \ldots, 6\}$. Then $D_k=FQ_{D_k}/J_{D_k}$, where
$Q_{D_k}$ is the quiver of $D_k$ and $J_{D_k}$ is an ideal in the
path algebra $FQ_{D_k}$ of $D_k$ over $F$. Since $D_1=D(F,
T^2_S)$, the quiver $Q_{D_1}$ contains a loop $C_1$ with $C^3_1$
not in $J_{D_1}$. Hence, for any $k \in \{1, \ldots, 6\}$, the
quiver $Q_{D_k}$ contains a cycle $C_k$ of length $k$ such that
the composition of any $3$ consecutive arrows in $C_k$ does not
belong to $J_{D_k}$. But then the cycle $C_k$ contains two
vertices $x$ and $y$ such that the shortest paths $u$ from $x$ to
$y$ and $v$ from $y$ to $x$ in $C_k$ are of length at most $3$,
and hence do not belong to $J_{D_k}$. Then we conclude that $\mod
D_k$  contains a short cycle
\[P_{D_k}(x) \rightarrow P_{D_k}(y) \rightarrow P_{D_k}(x),\]
with $P_{D_k}(x)$ and $P_{D_k}(y)$ being the indecomposable
projective modules in $\mod D_k$ at the vertices $x$ and $y$. In
particular, since $A=\widehat{B}/G=\widehat{B}/(\varrho^p)$, $\mod
A$ contains a short cycle of indecomposable projective modules.

Summing up, we have proved that (i) implies (ii).
\end{proof}

\section{Proof of the Theorem.}
Let $A$ be a selfinjective algebra over a field $K$. The
implication (ii)$\Rightarrow$(i) follows from Theorem \ref{3.4}.
We will show that (i) implies (ii). Assume that $\mod A$ has no
short cycles. If $A$ is a Nakayama algebra then the statement (ii)
follows from Proposition \ref{nowe4.1}. Hence, assume that $A$ is
not a Nakayama algebra. Applying Theorem \ref{4.1} and Lemma
\ref{4.A}, we conclude that $\Gamma_A$ admits a semiregular double
$\tau_A$-rigid stable slice $\Delta$. We note that $\Delta$ is a
Dynkin quiver, by the Riedtmann-Todorov theorem. Let $M$  be the
direct sum of the indecomposable modules lying on $\Delta$,
$I=r_A(M)$, and $B=A/I$. Then it follows from Theorem \ref{4.4}
that $H=\End_B(M)$ is a hereditary algebra of type $\Delta$,
$T=D(M)$ is a tilting module in $\mod H$ with $\End_H(T)$
isomorphic to $B$, and consequently $B$ is a tilted algebra of
Dynkin type $\Delta$. Further, applying Theorem \ref{4.4} again,
we obtain that $I$ is a deforming ideal of $A$ with $r_A(I)=eI$
for an idempotent $e$ of $A$, being a residual identity of
$B=A/I$, the algebra $A[I]$ is isomorphic to an orbit algebra
$\widehat{B}/(\psi\nu_{\widehat{B}})$ for a positive automorphism
$\psi$ of $\widehat{B}$, and $A$ is socle equivalent to $A[I]$. We
claim that $A$ is isomorphic to $A[I]$. Since $IeI=0$, it is
enough to show, by Proposition \ref{2.6}, that $e_i \neq
e_{\nu(i)}$ for any primitive summand $e_i$ of $e$,  where $\nu$
is the Nakayama permutation of $A$. Suppose that $e_i=e_{\nu(i)}$
for a primitive summand $e_i$ of $e$. Consider the indecomposable
projective module $P_i=e_iA$ in $\mod A$ given by $e_i$. Then, by
the definition of Nakayama permutation, the equality $e_i=e_{\nu
(i)}$ implies that $\top (P_i) \cong \soc (P_i)$. But then we have
in $\mod A$ a short cycle
\[\xymatrix{P_i \ar[r]^f & P_i \ar[r]^f & P_i}\]
where $f$ is a homomorphism whose image is $\soc (P_i)$, which
contradicts the assumption imposed on $A$. Therefore, $A$ is
isomorphic to $A[I]$, and consequently $A$ is isomorphic to the
orbit algebra $\widehat{B}/(\psi\nu_{\widehat{B}})$. Applying now
Theorem \ref{3.4}, we conclude that $A$ is isomorphic to an orbit
algebra $\widehat{B}/(\varphi\nu^2_{\widehat{B}})$ for a strictly
positive automorphism $\varphi$ of $\widehat{B}$. Summing up, we
have proved that (i) implies (ii).

\section*{Acknowledgements.}
The authors gratefully acknowledge support
 from the research grant DEC-2011/02/A/\ ST1/00216 of the National Science Center Poland.

The main result of the paper was presented by the first named
author during the conference "Advances in Representation Theory of
Algebras: Geometry and Homology" (CIRM Marseille-Luminy, September
2017).

\vspace{1cm} \noindent
Faculty of Mathematics and Computer Science\\
Nicolaus Copernicus University\\
Chopina 12/18\\
87-100 Toru\'n, Poland\\
E-mail: jaworska@mat.umk.pl (Alicja Jaworska-Pastuszak)

 \hspace{1.06cm}skowron@mat.uni.torun.pl (Andrzej Skowro\'nski)


\begin{thebibliography}{ABC}

\bibitem {ASS}
 I. Assem, D. Simson, A. Skowro\'nski, Elements of the Representation Theory of Associative Algebras 1: Techniques of Representation Theory, London Math. Soc. Stud. Texts, vol. 65, Cambridge Univ. Press, Cambridge, 2006.

\bibitem {ARS}
 M. Auslander, I. Reiten, S.O. Smal{\o}, Representation Theory of Artin Algebras, Cambridge Stud. Adv. Math., vol. 36, Cambridge Univ. Press, Cambridge, 1995.

\bibitem {BS1}
 M. B\l aszkiewicz, A. Skowro\'nski, On self-injective algebras of finite representation type, Colloq. Math. \textrm{127} (2012), 111--126.

\bibitem {BS2}
 M. B\l aszkiewicz, A. Skowro\'nski, On self-injective algebras of finite representation type with maximal almost split sequences, J. Algebra \textrm{422} (2015), 450--486.


\bibitem{BG}
 K. Bongartz, P. Gabriel, Covering Spaces in Representation-Theory, Invent. Math. 65 (1982), 331--378.


\bibitem {BLR}
O. Bretscher, C. L\"aser, C. Riedtmann, Selfinjective and simply
connected algebras, Manuscripta Math. \textrm{361} (1981),
253--307.

\bibitem {BS} R. Bautista, S.O. Smalø, Nonexistent cycles, Comm. Algebra \textrm{11} (1983),
1755--1767.

\bibitem {DR1}
 V. Dlab, C.M. Ringel, On algebras of finite representation type, J. Algebra \textrm{33} (1975), 306--394.

\bibitem {DR2}
 V. Dlab, C.M. Ringel, Indecomposable representations of graphs and algebras, Mem. Amer. Math. Soc. \textrm{6} (173) (1976).

\bibitem {G1}
 P. Gabriel, Unzerlegbare Darstellungen I, Manuscripta Math. \textrm{6} (1972), 71--103.

\bibitem {G2}
    P. Gabriel, The universal cover of a representation-finite algebra, in: Representations of Algebras, in: Lecture Notes in Math., vol. 903, Springer-Verlag, Berlin-Heidelberg, 1981, pp. 68--105.

\bibitem {HL} D. Happel, S. Liu, Module categories without short cycles are of finite type, Proc. Amer. Math. Soc. 120 (1994), 371--375.

\bibitem{Ho}
  M. Hoshino, Trivial extensions of tilted algebras, Comm. Algebra \textrm{10} (1982), 1965--1999.

\bibitem{HW}
  D. Hughes, J. Waschb\"usch, Trivial extensions of tilted algebras, Proc. London Math. Soc. \textrm{46} (1983), 347--364.

\bibitem{IT}
 K. Igusa, G. Todorov, A characterization of finite Auslander-Reiten quivers, J. Algebra \textrm{89} (1984), 148--177.

\bibitem{MP}
 R. Martinez-Villa, J.A. de la Pe\~na, Automorphisms of Representation Finite
 Algebras, Invent. Math. \textrm{72} (1983), 359--362.

\bibitem {RSS} I. Reiten, A. Skowro\'nski, S.O. Smal{\o}, Short chains and short cycles of modules, Proc. Amer. Math. Soc. 117 (1993), 343--354.

\bibitem{Rd1}
    C. Riedtmann, Algebren, Darstellungsk\"ocher, \"Uberlagerungen and zur\"uck, Comment. Math. Helv. \textrm{55} (1980), 199--224.

\bibitem{Rd2}
    C. Riedtmann, Representation-finite selfinjective algebras of class $\mathbb{A}_n$, in: Representation Theory II, in: Lecture Notes in Math., vol. 832, Springer-Verlag, Berlin-Heidelberg, 1980, pp. 449--520.

\bibitem{Rd3}
    C. Riedtmann, Representation-finite selfinjective algebras of class $\mathbb{D}_n$, Compos. Math. \textrm{49} (1983), 231--282.

\bibitem{S}
     A. Skowro\'nski, Selfinjective algebras: Finite and tame type, in: Trends in Representation Theory of Algebras and Related Topics, in: Contemp. Math. \textrm{406}, Amer. Math. Soc., Providence, RI, 2006, pp. 169--238.

\bibitem {SY1}
    A. Skowro\'nski, K. Yamagata, Socle deformations of self-injective algebras, Proc. London Math. Soc. \textrm{72} (1996), 545--566.

\bibitem {SY2}
    A. Skowro\'nski, K. Yamagata, Galois coverings of selfinjective algebras by repetitive algebras, Trans. Amer. Math. Soc. \textrm{351} (1999), 715--734.

\bibitem {SY3}
    A. Skowro\'nski, K. Yamagata, On invariability of selfinjective algebras of tilted type under stable equivalences, Proc. Amer. Math. Soc. \textrm{132} (2004), 659--667.

\bibitem {SY4}
    A. Skowro\'nski, K. Yamagata, Positive Galois coverings of selfinjective algebras, Adv. Math. \textrm{194} (2005), 398--436.

\bibitem {SY5}
    A. Skowro\'nski, K. Yamagata, Stable equivalences of selfinjective artin algebras of Dynkin type, Algebr. Represent. Theory \textrm{9} (2006), 33--45.

\bibitem {SY6}
    A. Skowro\'nski, K. Yamagata, Selfinjective algebras of quasitilted type, in: Trends in Representation Theory of Algebras and Related Topics, in: Eur. Math. Soc. Ser. Congr. Rep., European Math. Soc. Publ. House, Z\"urich, 2008, pp. 639--708.

\bibitem {SY7}
    A. Skowro\'nski, K. Yamagata, Frobenius Algebras I. Basic Representation Theory, European Math. Soc. Textbooks in Mathematics, European Math. Soc. Publ. House, Z\"urich, 2011.

\bibitem {SY7a}
    A. Skowro\'nski, K. Yamagata, Frobenius Algebras II. Tilted
    and Hochschild Extension Algebras, European Math. Soc. Textbooks in Mathematics, European Math. Soc. Publ. House, Z\"urich, 2017.

\bibitem {SY8}
    A. Skowro\'nski, K. Yamagata, On selfinjective algebras of tilted type, Colloq.
    Math. \textrm{141} (2015), 89--117.

\bibitem{T}
    G. Todorov, Almost split sequences for Tr$D$-periodic modules, in: Representation Theory II,
 Lecture Notes in Math., vol. 832, Springer-Verlag, Berlin-Heidelberg, 1980, pp. 600--631.

\bibitem{W1} J. Waschb\"usch, Symmetrische Algebren vom endlichen
Modultyp, J. Reine Angew. Math. \textrm{321} (1981), 78--98.

\bibitem{W2} J. Waschb\"usch, On selfinjective algebras of finite
representation type, Monographs of Institute of Mathematics, vol.
14, UNAM Mexico, 1983.

\bibitem{YX} R. Yang, J. Xiao, Construction of representation-finite selfinjective artin algebras of class
$B_n$ and $C_n$, Acta Math. Sinica, New Series, \textrm{9} (1993),
290--336.


\end{thebibliography}
\end{document}